\documentclass[a4paper,reqno]{amsart}

\usepackage{amssymb,eucal}

\newtheorem{theorem}{Theorem}[section]
\newtheorem{lemma}[theorem]{Lemma}
\newtheorem{corollary}[theorem]{Corollary}
\newtheorem{proposition}[theorem]{Proposition}

\numberwithin{equation}{section}

\theoremstyle{definition}
\newtheorem{definition}[theorem]{Definition}
\newtheorem*{example*}{Example}
\newtheorem{example}[theorem]{Example}

\newtheorem{remark}[theorem]{Remark}
\newtheorem*{remark*}{Remark}

\newcommand\frsl{\mathfrak{sl}}
\newcommand\frgl{\mathfrak{gl}}

\newcommand\frosp{\mathfrak{osp}}

\newcommand{\calL}{\mathcal{L}}
\newcommand{\calA}{\mathcal{A}}
\newcommand{\calB}{\mathcal{B}}
\newcommand{\calD}{\mathcal{D}}
\newcommand{\calT}{\mathcal{T}}
\newcommand{\calK}{\mathcal{K}}
\newcommand{\frg}{\mathfrak{g}}
\newcommand{\frb}{\mathfrak{b}}
\newcommand{\frs}{\mathfrak{s}}

\newcommand{\frh}{\mathfrak{h}}
\newcommand{\frm}{\mathfrak{m}}

\newcommand{\bF}{\mathbb{F}}

\newcommand{\bZ}{\mathbb{Z}}

\newcommand{\subo}{_{\bar 0}}
\newcommand{\subuno}{_{\bar 1}}
\newcommand{\id}{{\mathrm{id}}} 

\DeclareMathOperator{\End}{\mathrm{End}}
\DeclareMathOperator{\Mat}{\mathrm{Mat}} 
 \DeclareMathOperator{\ad}{ad}

\providecommand{\espan}[1]{\text{span}\left\{ #1\right\}}

\newenvironment{romanenumerate}
 {\begin{enumerate}
 
 }{\end{enumerate}}

\begin{document}

\title{Left unital Kantor triple systems and structurable algebras}


\author[Alberto Elduque]{Alberto Elduque$^{\star}$}
\thanks{$^{\star}$ Supported by the Spanish
Ministerio de Econom\'{\i}a y Competitividad and FEDER (MTM MTM2010-18370-C04-02) and by
the Diputaci\'on General de Arag\'on (Grupo de Investigaci\'on de \'Algebra)}
\address{Departamento de Matem\'aticas e
Instituto Universitario de Matem\'aticas y Aplicaciones, Universidad de
Zaragoza, 50009 Zaragoza, Spain} \email{elduque@unizar.es}

\author[Noriaki Kamiya]{Noriaki Kamiya$^{\star\star}$}
    \thanks{$^{\star\star}$ Supported by a Grant-in-aid for Scientific Research no.~19540042(C),(2), of the Japan Society for the Promotion of Science}
    \address{Department of Mathematics, University of Aizu, 965-8580, Aizuwakamatsu, Japan}
    \email{kamiya@u-aizu.ac.jp}

\author[Susumu Okubo]{Susumu Okubo$^{\star\star\star}$}
 \thanks{$^{\star\star\star}$ Supported in part by U.S.~Department of Energy Grant No. DE-FG02-91 ER40685.}
 \address{Department of Physics and Astronomy, University of
 Rochester, Rochester, NY 14627, USA}
 \email{okubo@pas.rochester.edu}

\date{May 7, 2012}

\begin{abstract}
Left unital Kantor triple systems will be shown to coincide with the triple systems attached to structurable algebras endowed with an involutive automorphism $\sigma$, with triple product given by the formula $xyz=V_{x,\sigma(y)}(z)$. A related result is proved for $(-1,-1)$ Freudenthal-Kantor triple systems. Some consequences for the associated $5$-graded Lie algebras and superalgebras are deduced too. In particular, left unital $(-1,-1)$ Freudenthal-Kantor triple systems are shown to be intimately related to Lie superalgebras graded over the root system of type $B(0,1)$.
\end{abstract}

\maketitle


\section{Introduction}

A \emph{structurable algebra} over a field is a unital (binary) algebra with involution $(A,\cdot,\bar{\ })$ that satisfies
\begin{gather}
(x-\bar x,y,z)=(y,\bar x- x,z), \label{eq:str1}\\
[V_{u,v},V_{x,y}]=V_{V_{u,v}x,y}-V_{x,V_{v,u}y}, \label{eq:str2}
\end{gather}
for any $u,v,x,y,z\in A$, where $(x,y,z)=(x\cdot y)\cdot z-x\cdot (y\cdot z)$ is the associator of $x,y,z$, and
\begin{equation}\label{eq:Vxy}
V_{x,y}(z)=(x\cdot \bar y)\cdot z+(z\cdot \bar y)\cdot x-(z\cdot \bar x)\cdot y.
\end{equation}
(See \cite{AllisonFaulkner93}.)

For fields of characteristic not two or three, it is known that \eqref{eq:str1} follows from \eqref{eq:str2} \cite{Allison78}.

On the other hand, a \emph{Kantor triple system} (or \emph{generalized Jordan triple system of second order} \cite{Kantor72,Kantor73}) is a vector space $U$ endowed with a trilinear map $U\times U\times U\rightarrow U$ satisfying:
\begin{gather}
[L(u,v),L(x,y)]=L(uvx,y)-L(x,vuy),\label{eq:KTS1}\\
K(K(u,v)x,y)=K(u,v)L(x,y)+L(y,x)K(u,v), \label{eq:KTS2}
\end{gather}
for any $u,v,x,y\in U$, where the maps $L(x,y)$ and $K(x,y)$ are given by
\[
L(x,y):z\mapsto xyz,\quad
K(x,y):z\mapsto xzy-yzx,
\]
for $x,y,z\in U$.

Given a structurable algebra $(A,\cdot,\bar{\ })$,
we may consider the triple product
\begin{equation}\label{eq:xyzVxyz}
\{xyz\}=V_{x,y}(z).
\end{equation}
Then $(A,\{xyz\})$ is a Kantor triple system \cite{Faulkner94}.

Moreover, if $e=1$ is the unity element of the structurable algebra $(A,\cdot,\bar{\ })$, then we have
\begin{gather}
\{eex\}=x, \label{eq:eex}\\
2\{xee\}+\{exe\}=3x, \label{eq:xeeexe}
\end{gather}
for any $x\in A$.

Conversely, over a field of characteristic not two or three, if $(U,\{xyz\})$ is a Kantor triple system,
and if $e$ is an element in $U$ satisfying \eqref{eq:eex} and \eqref{eq:xeeexe}, then Faulkner proved in \cite[Lemma 1.7]{Faulkner94} that $(U,\{xyz\})$ is the Kantor triple system obtained from a structurable algebra $(U,\cdot,\bar{\ })$ defined on $U$, with unity $e$, so that $\{xyz\}=(x\cdot \bar y)\cdot\bar z+(z\cdot \bar y)\cdot x-(z\cdot\bar x)\cdot y$ as in \eqref{eq:Vxy} and \eqref{eq:xyzVxyz}. (A different proof is given in \cite[Theorem 5.1]{KamiyaOkubo10}.)

In \cite{Faulkner94}, Faulkner studied a class of symmetric spaces called rotational, and he proved that they are intimately connected to real structurable algebras. To do so, he needed to prove first that if a Kantor triple system contains a distinguished element $e$ satisfying only \eqref{eq:eex} (we say then that $e$ is a \emph{left unit} of the Kantor triple system) then, assuming the characteristic is not two, three or five, still there exists a structurable algebra attached to the triple system, but in a indirect way. The definition of this structurable algebra is involved, and the triple product of the Kantor triple system does not have a clear description in terms of the binary product in the structurable algebra.

\smallskip

In this paper it will be proved that, over fields of characteristic not two or three, we may attach to any Kantor triple system with a left unit a structurable algebra $(A,\cdot,\bar{\ })$ endowed with an involutive automorphism $\sigma$, defined on the vector space of the Kantor triple system, such that now the triple product is given by the simple expression
\begin{equation}\label{eq:xyzVxsigmayz}
\begin{split}
xyz&=V_{x,\sigma(y)}(z),\\
    &=(x\cdot\sigma(\bar y))\cdot z+(z\cdot\sigma(\bar y))\cdot x-(z\cdot \bar x)\cdot\sigma(y),
\end{split}
\end{equation}
and conversely.
That is, the Kantor triple systems with a left unit are precisely the Kantor triple systems obtained from structurable algebras $(A,\cdot,\bar{\ })$ endowed with an involutive automorphism $\sigma$, so that the triple product is recovered by \eqref{eq:xyzVxsigmayz}.

Moreover, the arguments used for left unital Kantor triple systems can be used to study left unital $(-1,-1)$ Freudenthal-Kantor triple systems. Besides, the $5$-graded Lie algebra naturally attached to any left unital Kantor triple system is shown to be graded over the nonreduced root system $BC_1$, and of type $B_1$, while the Lie superalgebra attached to any left unital $(-1,-1)$ Freudenthal-Kantor triple system is graded over the root system of the simple Lie superalgebra $B(0,1)$. Conversely, any $B(0,1)$-graded Lie superalgebra gives rise to a left unital $(-1,-1)$ Freudenthal-Kantor triple system.

It should be mentioned here that the case of the generalized Jordan triple systems containing an element $e$ satisfying $eex=xee=x$ instead of \eqref{eq:eex} and \eqref{eq:xeeexe} has been studied to be intimately related with a variety of nonconmutative Jordan algebras in \cite{EKO-1-1ncJa}. Some balanced $(-1,-1)$ Freudenthal-Kantor triple systems are then related to the algebras in this variety that have degree two.

Throughout the paper we fix a ground field $\bF$ of characteristic not two or three.


\section{Left unital Kantor triple systems and structurable algebras with involutive automorphisms}

Let us first recall some definitions.
Given a triple system $(U,xyz)$ we will denote by $L(x,y)$  the linear operator on $U$ given by
$
L(x,y)z=xyz
$,
for $x,y,z\in U$.

\begin{definition}[\cite{Kantor72}] \null\quad\null
\begin{romanenumerate}
\item
A triple system $(U,xyz)$ is said to be a \emph{generalized Jordan triple system} if it satisfies
\begin{equation}\label{eq:GJTS}
uv(xyz)=(uvx)yz-x(vuy)z+xy(uvz)
\end{equation}
for any $u,v,x,y,z\in U$ or, equivalently,
\begin{equation}\label{eq:GJTSbis}
[L(u,v),L(x,y)]=L(L(u,v)x,y)-L(x,L(v,u)y),
\end{equation}
for $u,v,x,y\in U$.

\item A generalized Jordan triple system $(U,xyz)$ is said to be \emph{left unital} if there is an element $e\in U$ such that $L(e,e)=\id$ (i.e., $eex=x$ for any  $x\in U$). The element $e$ is said to be a \emph{left unit}.

\item A generalized Jordan triple system $(U,xyz)$ is called a \emph{Kantor triple system} if it satisfies
\begin{equation}\label{eq:KTS}
K(K(u,v)x,y)=K(u,v)L(x,y)+L(y,x)K(u,v),
\end{equation}
for any $u,v,x,y\in U$, where
\begin{equation}\label{eq:K}
K(x,y)z=xzy-yzx.
\end{equation}
\end{romanenumerate}
\end{definition}

\begin{remark} The left unit is not necessarily unique, as shown by the generalized Jordan triple system $(U,xyz)$, with $xyz=(x\vert y)z$ for a symmetric bilinear form. Any element $u$ with $(u\vert u)=1$ is a left unit here. \qed
\end{remark}

Given a left unital generalized Jordan triple system $(U,xyz)$, fix a left unit $e$ and consider the linear maps:
\begin{equation}\label{eq:rho_mu}
\begin{cases}
\rho:U\rightarrow U,\ x\mapsto xee,\\
\mu:U\rightarrow U,\ x\mapsto exe.
\end{cases}
\end{equation}

\begin{lemma}\label{le:rhomu}
Let $(U,xyz)$ be a left unital generalized Jordan triple system with left unit $e$. Then, for any $x,y\in U$, the following conditions hold
\begin{gather}
L(xye,e)=L(e,yxe), \label{eq:1a}\\
L(\rho(x),e)=L(e,\mu(x)),\quad L(\mu(x),e)=L(e,\rho(x)), \label{eq:1b}\\
\rho^2=\mu^2,\quad \rho\mu=\mu\rho. \label{eq:1c}
\end{gather}
\end{lemma}
\begin{proof}
Equation \eqref{eq:GJTSbis} with $x=y=e$, so that $L(x,y)=\id$, gives \eqref{eq:1a}. Now \eqref{eq:1b} is obtained by imposing $x=e$ or $y=e$ in \eqref{eq:1a}. Also, \eqref{eq:1b} implies $\rho(x)ee=e\mu(x)e$ and $\mu(x)ee=e\rho(x)e$, or $\rho^2=\mu^2$ and $\rho\mu=\mu\rho$.
\end{proof}

\begin{lemma}\label{le:rho13}
Let $(U,xyz)$ be a left unital Kantor triple system with left unit $e$. Then, for any $u,v\in U$, the following conditions hold:
\begin{gather}
K(u,e)e=(\rho-\id)(u),\label{eq:2a}\\
K(u,v)=\frac{1}{2}K(K(u,v)e,e), \label{eq:2b}\\
(\rho-\id)(\rho-3\id)=0. \label{eq:2c}
\end{gather}
In particular, $\rho$ and $\mu$ are invertible linear operators.
\end{lemma}
\begin{proof}
Equation \eqref{eq:2a} follows from the definition of the operators $K(x,y)$ in \eqref{eq:K}, while \eqref{eq:2b} is a straightforward consequence of \eqref{eq:KTS}. Now, from \eqref{eq:2a} and \eqref{eq:2b} we get
\[
(\rho-\id)(u)=K(u,e)e=\frac12 K(K(u,e)e,e)  
=\frac12(\rho-\id)K(u,e)e=\frac12(\rho-\id)^2(u),
\]
whence \eqref{eq:2c}.
\end{proof}

Therefore, if $(U,xyz)$ is a left unital Kantor triple system with left unit $e$, then
\[
U=U_1\oplus U_3,
\]
with $U_i=\{x\in U: \rho(x)=ix\}$, $i=1,3$. Define the binary product $\star$ on $U$ by means of
\begin{equation}\label{eq:binaryproduct}
x\star y=e\mu^{-1}(x)y
\end{equation}
for $x,y\in U$.

\begin{lemma}\label{le:xcdoty}
With these conventions,  we have:
\begin{romanenumerate}
\item $e$ is the unity of the algebra $(U,\star)$: $e\star x=x=x\star e$ for any $x\in U$.

\item  For any $x,y\in U$
\begin{equation}\label{eq:exyxey}
exy=\mu(x)\star y,\qquad xey=\rho(x)\star y.
\end{equation}

\item For any $x,y,z\in U$
\begin{equation}\label{eq:xyzmurho}
xyz=x\star (\mu(y)\star z)-\mu(y)\star (x\star z)+\mu\bigl(\mu\rho^{-1}(x)\star y\bigr)\star z.
\end{equation}

\item For any $x,y\in U$ we have $xye=\Lambda_{x,y}(eyx)$, with
    \begin{equation}\label{eq:Lambdaxy}
    \Lambda_{x,y}=\begin{cases}
    \id&\text{for $x\in U_1$,}\\
    -\hat\rho&\text{for $x\in U_3$, $y\in U_1$,}\\
    \frac{1}{3}\hat\rho&\text{for $x,y\in U_3$,}
    \end{cases}
    \end{equation}
    where $\hat \rho=3\id -2\rho$.

\item For any $x,y\in U$,
    \begin{equation}\label{eq:rhoxcdoty}
    \rho(x\star y)=\begin{cases}
    2x\star y-y\star x,&\text{if $x,y\in U_1$ or $x,y\in U_3$,}\\
    3y\star x,&\text{if $x\in U_1$ and $y\in U_3$,}\\
    4x\star y-y\star x,&\text{if $x\in U_3$ and $y\in U_1$.}
    \end{cases}
    \end{equation}

\end{romanenumerate}
\end{lemma}
\begin{proof}
The assertion in (i) is clear. For (ii) note that $\mu(x)\star y=e\bigl(\mu^{-1}\mu(x)\bigr)y=exy$, and
\[
\rho(x)\star y=e(\mu^{-1}\rho(x))y=e(\rho\mu^{-1}(x))y=L(e,\rho\mu^{-1}(x))(y)=L(x,e)y=xey,
\]
because of \eqref{eq:1b}.

For (iii) we start with
\[
ex(eyz)=(exe)yz-e(xey)z+ey(exz),
\]
which we rewrite, using (ii), as
\[
\mu(x)\star\bigl(\mu(y)\star z\bigr)=\mu(x)yz-\mu(\rho(x)\star y)\star z+\mu(y)\star(\mu(x)\star z),
\]
which is equivalent to the assertion in (iii).

For (iv), if $x\in U_1$, \eqref{eq:2a} gives $K(x,e)e=0$, and hence $K(x,e)=0$ by \eqref{eq:2b} and $xye=K(x,e)y+eyx=eyx$.
However, if $x\in U_3$, \eqref{eq:2a} gives $K(x,e)e=2x$ and $K(K(x,e)y,e)e$ equals
\[
\begin{cases}
(\rho-\id)K(x,e)y,\quad\text{by \eqref{eq:2b},}&\\
K(x,e)L(y,e)e+L(e,y)K(x,e)e=K(x,e)\rho(y)+2eyx,\quad\text{by \eqref{eq:KTS}.}&
\end{cases}
\]
Thus $(\rho-\id)K(x,e)y=2eyx+K(x,e)\rho(y)$ for any $y\in U$.

Now, if $y\in U_1$ this gives $(\rho-2\id)K(x,e)y=2eyx$, but $(\rho-2\id)^2=\id$ by \eqref{eq:2c}, so we obtain $xye-eyx=K(x,e)y=2(\rho-2\id)(eyx)$, or $xye=(2\rho-3\id)(eyx)=-\hat\rho(eyx)$. If $y\in U_3$ the same argument gives $(\rho-4\id)K(x,e)y=2eyx$, and $\rho(\rho-4\id)=-3\id$, so we get $xye-eyx=K(x,e)y=-\frac{1}{3}2\rho(eyx)$ or $xye=\frac{1}{3}(3\id-2\rho)(eyx)=\frac13\hat\rho(eyx)$.

For (v) we start with
\[
ex(yee)=(exy)ee-y(xee)e+ye(exe),
\]
which we rewrite as
\[
ex\rho(y)=\rho(exy)-y\rho(x)e+ye\mu(x).
\]
With $x\mapsto \mu^{-1}(x)$ and using (ii), this is equivalent to
\[
x\star \rho(y)=\rho(x\star y)+\rho(y)\star x-y\rho\mu^{-1}(x)e.
\]
Using (ii) and (iii) this gives:
\[
\rho(x\star y)=x\star\rho(y)-\rho(y)\star x+\begin{cases}
\rho(x)\star y&\text{if $y\in U_1$,}\\
-\hat\rho(\rho(x)\star y),&\text{if $x\in U_1$ and $y\in U_3$,}\\
\frac13\hat\rho(\rho(x)\star y),&\text{if $x,y\in U_3$.}
\end{cases}
\]
If $x,y\in U_1$ we get $\rho(x\star y)=x\star y-y\star x+x\star y=2x\star y-y\star x$. If $x,y\in U_3$ we get $\rho(x\star y)=3x\star y-3y\star x+\hat\rho(x\star y)$, that is, $\rho(x\star y)=2x\star y -y\star x$ again. If $x\in U_1$ and $y\in U_3$ we get $\rho(x\star y)=3x\star y-3y\star x-\hat\rho(x\star y)$, that is, $\rho(x\star y)=3y\star x$. Finally, if $x\in U_3$ and $y\in U_1$, we obtain $\rho(x\star y)=x\star y-y\star x-3x\star y=4x\star y-y\star x$.
\end{proof}

\begin{proposition}\label{pr:sigma}
Let $(U,xyz)$ be a left unital Kantor triple system with left unit $e$, and consider the linear map $\sigma=\mu^{-1}\hat\rho$ ($\hat\rho=3\id-2\rho$ as above). Then $\sigma$ is an involutive automorphism of $(U,xyz)$.
\end{proposition}
\begin{proof}
$\sigma$ commutes with $\mu$ and $\rho$ and $\sigma^2=\mu^{-2}(3\id-2\rho)^2=\rho^{-2}(9\id-12\rho+4\rho^2)
=\rho^{-2}(3(\rho-\id)(\rho-3\id)+\rho^2)=\id$ by \eqref{eq:1c} and \eqref{eq:2c}, and hence \eqref{eq:xyzmurho} shows that it is enough to prove that $\sigma$ is an automorphism of the algebra $(U,\star)$ defined by \eqref{eq:binaryproduct}. Therefore we must prove
\[
\sigma(\mu(x)\star\mu(y))=\sigma(\mu(x))\star\sigma(\mu(y))
\]
or
\begin{equation}\label{eq:rhohat_mu}
\hat\rho\bigl(\mu(x)\star\mu(y)\bigr)=\mu\bigl(\hat\rho(x)\star\hat\rho(y)\bigr)
\end{equation}
for $x,y\in U$.

Equation \eqref{eq:1a} gives $(xye)ee=e(yxe)e$, or $\rho(xye)=\mu(yxe)$, which is equivalent (Lemma \ref{le:xcdoty}) to
$
\rho\Lambda_{x,y}(\mu(y)\star x)=\mu\Lambda_{y,x}(\mu(x)\star y)
$. Change $x$ to $\mu(x)$ and use \eqref{eq:1c} to obtain
\begin{equation}\label{eq:rhoLambda}
\rho\Lambda_{x,y}\bigl(\mu(y)\star\mu(x)\bigr)=\mu\Lambda_{y,x}\bigl(\rho^2(x)\star y\bigr).
\end{equation}
Set $\alpha(x)=i$ if $x\in U_i$, $i=1,3$, and $\beta(x)=1$ for $x\in U_1$ and $\beta(x)=-3$ for $x\in U_3$. Hence $\rho(x)=\alpha(x)x$ and $\hat\rho(x)=\beta(x)x$ for any $x\in U_1\cup U_3$. Then \eqref{eq:rhoLambda} can be rewritten as
\[
\begin{split}
\rho\Lambda_{x,y}\bigl(\mu(y)\star\mu(x)\bigr)&=
\frac{\alpha^2(x)}{\beta(x)\beta(y)}\Lambda_{y,x}\mu\bigl(\hat\rho(x)\star\hat\rho(y)\bigr)\\
&=\frac{\beta(x)}{\beta(y)}\Lambda_{y,x}\mu\bigl(\hat\rho(x)\star\hat\rho(y)\bigr),
\end{split}
\]
for $x,y\in U_1\cup U_3$, because $\alpha^2(x)=\beta^2(x)$. Then, \eqref{eq:rhohat_mu} is satisfied if and only if
\[
\hat\rho(x\star y)=\frac{\beta(y)}{\beta(x)}\rho\Lambda_{y,x}^{-1}\Lambda_{x,y}(y\star x),
\]
or
\begin{equation}\label{eq:yx_rhos}
y\star x=\frac{\beta(x)}{\beta(y)}\Lambda_{x,y}^{-1}\Lambda_{y,x}\rho^{-1}\hat\rho(x\star y),
\end{equation}
for $x,y\in U_1\cup U_3$.

Equation \eqref{eq:2c} gives $\rho^{-1}\hat\rho=\rho^{-1}(3\id-2\rho)=2\id-\rho$. Also \eqref{eq:Lambdaxy} gives
\[
\Lambda_{x,y}^{-1}\Lambda_{y,x}=\begin{cases}
\id,&\text{for $x,y\in U_1$ or $x,y\in U_3$,}\\
-\hat\rho,&\text{for $x\in U_1$, $y\in U_3$,}\\
-\hat\rho^{-1},&\text{for $x\in U_3$, $y\in U_1$.}
\end{cases}
\]
Hence,
\[
\frac{\beta(x)}{\beta(y)}\Lambda_{x,y}^{-1}\Lambda_{y,x}\rho^{-1}\hat\rho=
\begin{cases}
2\id-\rho,&\text{for $x,y\in U_1$ or $x,y\in U_3$,}\\
\frac{1}{3}\hat\rho(2\id-\rho)=\frac{1}{3}\rho,&\text{for $x\in U_1$, $y\in U_3$,}\\
3\hat\rho^{-1}\rho^{-1}\hat\rho=3\rho^{-1},&\text{for $x\in U_3$, $y\in U_1$,}
\end{cases}
\]
and \eqref{eq:yx_rhos} follows at once from Lemma \ref{le:xcdoty}(v).
\end{proof}

\begin{lemma}\label{le:UUsigma}
Let $(U,xyz)$ be a Kantor triple system with an involutive automorphism $\sigma$. Define a new triple product by $\{xyz\}=x\sigma(y)z$. Then $(U,\{xyz\})$ is a Kantor triple system too, and $\sigma$ is an automorphism of $(U,\{xyz\})$.
\end{lemma}
\begin{proof}
Denote by $\tilde L(x,y):z\mapsto \{xyz\}$, and $\tilde K(x,y):z\mapsto \{xzy\}-\{yzx\}$, the $L$ and $K$ operators of the new triple system. Then $\tilde L(x,y)=L(x,\sigma(y))$ and $\tilde K(x,y)=K(x,y)\sigma$. For any $u,v,x,y\in U$ we have:
\[
\begin{split}
[\tilde L(u,v),\tilde L(x,y)]&=
  [L(u,\sigma(v)),L(x,\sigma(y))]\\
  &=L(u\sigma(v)x,\sigma(y))-L(x,\sigma(v)u\sigma(y))\\
  &=\tilde L(\{uvx\},y)-\tilde L(x,\{vuy\}),
\end{split}
\]
and
\[
\begin{split}
\tilde K(\tilde K(u,v)x,y)
  &=K(K(u,v)\sigma(x),y)\sigma\\
  &=K(u,v)L(\sigma(x),y)\sigma+L(y,\sigma(x))K(u,v)\sigma\\
  &=K(u,v)\sigma L(x,\sigma(y))+L(y,\sigma(x))K(u,v)\sigma\\
  &=\tilde K(u,v)\tilde L(x,y)+\tilde L(y,x)\tilde K(u,v),
\end{split}
\]
because $\sigma L(x,y)=L(\sigma(x),\sigma(y))\sigma$ for any $x,y$ and $\sigma^2=\id$. Hence $(U,\{xyz\})$ is a Kantor triple system. The assertion on $\sigma$ being an automorphism of $(U,\{xyz\})$ is clear.
\end{proof}

\begin{lemma}\label{le:UUsigma_e}
Let $(U,xyz)$ be a left unital Kantor triple system with left unit $e$, and let $\sigma$ be the involutive automorphism $\sigma=\mu^{-1}(3\id-\rho)$ as in Proposition \ref{pr:sigma}. Consider the new Kantor triple system $(U,\{xyz\})$ with $\{xyz\}=x\sigma(y)z$ as in Lemma \ref{le:UUsigma}. Then $\{eex\}=x$ and $2\{xee\}+\{exe\}=3x$ for any $x\in U$.

Conversely, let $(U,\{xyz\})$ be a Kantor triple system containing an element $e$ such that $\{eex\}=x$ and $2\{xee\}+\{exe\}=3x$ for any $x\in U$, and endowed with an involutive automorphism $\sigma$ such that $\sigma(e)=e$. Define a new triple product on $U$ by $xyz=\{x\sigma(y)z\}$. Then $(U,xyz)$ is a left unital Kantor triple system with left unit $e$, $\sigma$ is an automorphism of $(U,xyz)$, and $\sigma=\mu^{-1}(3\id-2\rho)$, for $\rho$ and $\mu$ given in \eqref{eq:rho_mu}.
\end{lemma}
\begin{proof}
For the first part, note that $\{eex\}=eex=x$ as $\sigma(e)=e$, and $2\{xee\}+\{exe\}=2xee+e\sigma(x)e=(2\rho-\mu\sigma)(x)=(2\rho-(3\id-2\rho))(x)=3x$ for any $x\in U$.

For the converse, Lemma \ref{le:UUsigma} shows that $(U,xyz)$ is a Kantor triple system with involutive automorphism $\sigma$. Besides $eex=\{eex\}=x$, because $\sigma(e)=e$. Moreover, $3x=\{xee\}+\{exe\}=2xee+e\sigma(x)e=(2\rho+\mu\sigma)(x)$ for any $x\in U$. Hence $2\rho+\mu\sigma=3\id$ and $\sigma=\mu^{-1}(3\id-2\rho)$.
\end{proof}

We arrive to the main result of the paper:

\begin{theorem}\label{th:main}
Let $(U,xyz)$ be a left unital Kantor triple system with left unit $e$. Define a linear map and a (binary) multiplication on $U$ by means of
\begin{equation}\label{eq:leftKantor_structurable}
\begin{cases}
\bar x =2x-xee,&\\
x\cdot y=\bar xey-\bar x\sigma(\bar y)e+yex,&
\end{cases}
\end{equation}
for $x,y\in U$. Then $(U,\cdot,\bar{\ }))$ is a structurable algebra with unity $e$, $\sigma=\mu^{-1}(3\id-2\rho)$ is an involutive automorphism of $(U,\cdot,\bar{\ })$ and the triple product on $U$ is recovered as in equation \eqref{eq:xyzVxsigmayz}:
\begin{equation}\label{eq:triple_binary}
xyz=(x\cdot\sigma(\bar y))\cdot z+(z\cdot\sigma(\bar y))\cdot x-(z\cdot \bar x)\cdot\sigma(y),
\end{equation}
for any $x,y,z\in U$.

Conversely, if $(A,\cdot,\bar{\ })$ is a structurable algebra endowed with an involutive automorphism $\sigma$, and $e=1$ is the unity of $A$, then $(A,xyz)$ is a left unital Kantor triple system with left unit $e$, where $xyz$ is defined by the formula in \eqref{eq:triple_binary}. Moreover, $\sigma=\mu^{-1}(3\id-2\rho)$, and the involution and the multiplication in the structurable algebra are related to this triple product by equation \eqref{eq:leftKantor_structurable}.
\end{theorem}
\begin{proof}
If $(U,xyz)$ is a left unital Kantor triple system with left unit $e$, Lemma \ref{le:UUsigma_e} shows that with $\sigma=\mu^{-1}(3\id-2\rho)$ and $\{xyz\}=x\sigma(y)z$, $(U,\{xyz\})$ is a Kantor triple system with $\{eex\}=x$, $2\{xee\}+\{exe\}=3x$, and where $\sigma$ is an involutive automorphism  such that $\sigma(e)=e$. Now \cite[Lemma 1.7]{Faulkner94} (see also \cite[Theorem 5.1]{KamiyaOkubo10}) proves that $(U,\cdot,\bar{\ })$ is a structurable algebra with unity $e$, involution $\bar x=2x-\{xee\}=2x-xee$ and multiplication $x\cdot y=\frac{1}{2}(\{xey\}+\{exy\}+\{xye\}-\{eyx\})$. Then \cite[(24--28)]{Faulkner94} show that in this case the triple product $\{xyz\}$ is recovered as
\[
\{xyz\}=(x\cdot\bar y)\cdot z+(z\cdot\bar y)\cdot x-(z\cdot \bar x)\cdot y.
\]
In particular
\[
\begin{split}
\{\bar xey\}&=\bar x\cdot y+y\cdot \bar x-y\cdot x,\\
\{\bar x\bar y e \}&=\bar x\cdot y+ y\cdot\bar x-x\cdot\bar y,\\
\{yex\}&=y\cdot x+ x\cdot y-x\cdot \bar y,
\end{split}
\]
so that
\[
\bar xey-\bar x\sigma(\bar y)e+yex
 =\{\bar xey\}-\{\bar x\bar y e\}+\{yex\}\\
 =x\cdot y,
\]
for any $x,y\in U$, and
\[
xyz=\{x\sigma(y)z\}=(x\cdot\sigma(\bar y))\cdot z+(z\cdot\sigma(\bar y))\cdot x-(z\cdot \bar x)\cdot\sigma(y),
\]
for any $x,y,z\in U$.

Conversely, if $(A,\cdot,\bar{\ })$ is a structurable algebra, \cite[Lemma 1.7]{Faulkner94} or \cite[Theorem 5.1]{KamiyaOkubo10} show that with the triple product $\{xyz\}=V_{x,y}(z)=(x\cdot\bar y)\cdot z+(z\cdot\bar y)\cdot x-(z\cdot \bar x)\cdot y$ for $x,y,z\in A$, $(A,\{xyz\})$ is a Kantor triple system where the unity $e=1$ satisfies $\{eex\}=x$, $2\{xee\}+\{exe\}=3x$ for any $x\in A$, and the involution and binary multiplication on $A$ are recovered as follows: $\bar x=2x-\{xee\}$ and $x\cdot y=\{\bar xey\}-\{\bar x\bar y e\}+\{yex\}$ as above. Now, if $\sigma$ is an involutive automorphism of $(A,\cdot,\bar{\ })$, then Lemma \ref{le:UUsigma_e} shows that $(A,xyz)$ is a left unital Kantor triple system with left unit $e$, where $xyz=\{x\sigma(y)z\}$, and where $\sigma=\mu^{-1}(3\id-2\rho)$. The result follows.
\end{proof}

\begin{remark} Let $(U,xyz)$ be a left unital Kantor triple system with left unit $e$, and let $(U,\cdot,\bar{\ })$ be the structurable algebra with involution and multiplication given by \eqref{eq:leftKantor_structurable}. Then the eigenspaces for $\rho:x\mapsto xee$ are precisely the subspaces of symmetric and skew-symmetric elements for the involution:
\[
\begin{split}
U_1&=\{x\in U:xee=x\}=\{x\in U: \bar x=x\},\\
U_3&=\{x\in U: xee=3x\}=\{x\in U: \bar x=-x\}. \qed
\end{split}
\]
\end{remark}


\section{Left unital $(-1,-1)$-Freudenthal-Kantor triple systems}

In \cite{YO84}, Yamaguti and Ono considered a wide class of triple systems: the $(\epsilon,\delta)$ Freudenthal-Kantor triple systems, which are useful tools in the construction of Lie algebras and superalgebras.

An $(\epsilon,\delta)$ \emph{Freudental-Kantor triple system} ($\epsilon,\delta$ are either $1$ or $-1$) is a triple system $(U,xyz)$ such that, if $L(x,y),$ and $K(x,y)$ are given by
\[
L(x,y):z\mapsto xyz,\quad
K(x,y):z\mapsto xzy-\delta yzx,
\]
then
\begin{subequations}\label{eq:FK}
\begin{gather}
[L(u,v),L(x,y)]=L(uvx,y)+\epsilon L(x,vuy),\label{eq:FK1}\\
K\bigl(K(u,v)x,y\bigr)=L(y,x)K(u,v)-\epsilon K(u,v)L(x,y),\label{eq:FK2}
\end{gather}
hold for any $x,y,u,v\in U$.
\end{subequations}

Kantor triple systems are exactly the $(-1,1)$ Freudenthal-Kantor triple systems. Actually, for $\epsilon=-1$, \eqref{eq:FK1} and \eqref{eq:FK2} coincide with \eqref{eq:KTS1} and \eqref{eq:KTS2} (or \eqref{eq:GJTS} and \eqref{eq:KTS}), but $K(x,y)$ is symmetric on $x$ and $y$ for $\delta=-1$, and alternating for $\delta=1$. In particular, $(-1,-1)$ Freudenthal-Kantor triple systems are generalized Jordan triple systems. Also, Freudenthal triple systems, symplectic triple systems and Faulkner ternary algebras are intimately related to $(1,1)$ Freudenthal-Kantor triple systems. (See, for instance, \cite[Theorem 4.7]{ElduqueMagic} and \cite[Theorem 2.18]{ElduqueNewSimple3}, and references therein, for the relationship between these triple systems.)

With the same arguments as in Lemma \ref{le:UUsigma} we have:

\begin{lemma}\label{le:UUsigma_bis}
Let $(U,xyz)$ be an $(\epsilon,\delta)$ Freudenthal-Kantor triple system endowed with an automorphism $\sigma$ such that $\sigma^2=\pm\id$. Define a new triple product by $\{xyz\}=x\sigma(y)z$. Then $(U,\{xyz\})$ is a $(\pm\epsilon,\delta)$ Freudenthal-Kantor triple system, and $\sigma$ is an automorphism of $(U,\{xyz\})$ too. \qed
\end{lemma}

\begin{corollary}\label{co:UUsigma_bis}
Let $(U,xyz)$ be an $(\epsilon,\delta)$ Freudenthal-Kantor triple system endowed with a bijective linear map $\sigma:U\rightarrow U$ satisfying:
\[
\sigma^2=\pm\mu\id,\quad \sigma(x)\sigma(y)\sigma(z)=\mu\sigma(xyz),
\]
for any $x,y,z\in U$, where $0\ne \mu\in \bF$. Define a new triple product by $\{xyz\}=x\sigma(y)z$. Then $(U,\{xyz\})$ is a $(\pm\epsilon,\delta)$ Freudenthal-Kantor triple system.
\end{corollary}
\begin{proof}
By extending scalars if necessary, the map $\tilde\sigma:x\mapsto \sqrt{\mu}^{-1}\sigma(x)$ is an automorphism of $(U,xyz)$ with $\tilde\sigma^2=\pm\id$ and Lemma \ref{le:UUsigma_bis} applies.
\end{proof}

\begin{example}\label{ex:UUsigma_bis}
Let $(U,xyz)$ be an $(\epsilon,\delta)$ Freudenthal-Kantor triple system. Consider the $(\epsilon,\delta)$ Freudenthal-Kantor triple system defined on \[
M_{2,1}(U)=\left\{\begin{pmatrix} x\\ y\end{pmatrix}: x,y\in U\right\},
\]
with componentwise multiplication. Then the map
\[
\sigma:\begin{pmatrix} x\\ y\end{pmatrix}\mapsto \begin{pmatrix} y\\ -x\end{pmatrix},
\]
is an automorphism of $(M_{2,1}(U),XYZ)$ with $\sigma^2=-\id$. Hence with $\{XYZ\}=X\sigma(Y)Z$, for $X,Y,Z\in M_{2,1}(U)$, $M_{2,1}(U)$ becomes a $(-\epsilon,\delta)$ Freudenthal-Kantor triple system. \qed
\end{example}

\begin{proposition}\label{pr:sigma_in_K}
Let $(U,xyz)$ be an $(\epsilon,\delta)$ Freudenthal-Kantor triple system, and let $\sigma\in K(U,U)$ (the linear span of the operators $K(x,y)$ for $x,y\in U$) satisfying $\sigma^2=\epsilon\delta\id$. Then $\sigma$ is an automorphism of $(U,xyz)$. Therefore, with the new triple product defined by $\{xyz\}=x\sigma(y)z$, $(U,\{xyz\})$ is a $(\delta,\delta)$ Freudenthal-Kantor triple system.
\end{proposition}
\begin{proof}
Equation \eqref{eq:FK2} proves
\begin{equation}\label{eq:Ksigmaxy}
K(\sigma(x),y)=L(y,x)\sigma-\epsilon\sigma L(x,y)
\end{equation}
for any $x,y\in U$. Therefore, we have
\[
\sigma(x)zy-\delta yz\sigma(x)=yx\sigma(z)-\epsilon\sigma(xyz).
\]
In other words, the equation
\begin{equation}\label{eq:sigma_xyz}
\sigma(xyz)=-\epsilon\sigma(x)zy+\epsilon yx\sigma(z)+\epsilon\delta yz\sigma(x)
\end{equation}
holds for any $x,y,z\in U$.

Apply $\sigma$ to both sides of \eqref{eq:sigma_xyz} to get
\begin{equation}\label{eq:sigma_xyz_bis}
\epsilon\delta xyz=-\epsilon\sigma(\sigma(x)zy)+\epsilon\sigma(yx\sigma(z))
+\epsilon\delta\sigma(yz\sigma(x)),
\end{equation}
which, using \eqref{eq:sigma_xyz} on each summand on the right hand side, becomes:
\[
\begin{split}
\epsilon\delta xyz&
    =-\epsilon\Bigl(-\delta xyz+\epsilon z\sigma(x)\sigma(y)+ zyx\Bigr)\\
    &\qquad +\epsilon\Bigl(-\epsilon\sigma(y)\sigma(z)x+\delta xyz+\epsilon\delta x\sigma(z)\sigma(y)\Bigr)\\
    &\qquad +\epsilon\delta\Bigl(-\epsilon\sigma(y)\sigma(x)z+\delta zyx+\epsilon\delta z\sigma(x)\sigma(y)\Bigr)\\[4pt]
    &=2\epsilon\delta xyz-\sigma(y)\sigma(z)x+\delta x\sigma(z)\sigma(y)-\delta\sigma(y)\sigma(x)z,
\end{split}
\]
and hence
\[
\epsilon\delta xyz=-\delta x\sigma(z)\sigma(y)+\delta\sigma(y)\sigma(x)z+\sigma(y)\sigma(z)x.
\]
With the substitutions $x\mapsto \sigma(x)$, $y\mapsto \sigma(y)$ and $z\mapsto \sigma(z)$, we obtain,
\[
\begin{split}
\sigma(x)\sigma(y)\sigma(z)
    &=\epsilon\delta\Bigl(-\delta\sigma(x)zy+\delta yx\sigma(z)+yz\sigma(x)\Bigr)\\
    &=-\epsilon\sigma(x)zy+\epsilon yx\sigma(z)+\epsilon\delta yz\sigma(x)\\
    &=\sigma(xyz)\qquad \text{(using \eqref{eq:sigma_xyz}).}
\end{split}
\]
This shows that $\sigma$ is an automorphism of $(U,xyz)$. The last assertion follows at once from Lemma \ref{le:UUsigma_bis}.
\end{proof}

\begin{corollary}\label{co:sigma_in_K}
Let $(U,xyz)$ be an $(\epsilon,\delta)$ Freudenthal-Kantor triple system, and let $\sigma\in K(U,U)$ (the linear span of the operators $K(x,y)$ for $x,y\in U$) satisfying $\sigma^2=\mu\id$ for a nonzero scalar $\mu\in \bF$. Then $\sigma(xyz)=\epsilon\delta\mu^{-1}\sigma(x)\sigma(y)\sigma(z)$ for any $x,y,z\in U$.
Moreover, with the new triple product defined by $\{xyz\}=x\sigma(y)z$, $(U,\{xyz\})$ is a $(\delta,\delta)$ Freudenthal-Kantor triple system. \qed
\end{corollary}
\begin{proof}
As in the proof of Corollary \ref{co:UUsigma_bis}, extend scalars if necessary and consider the map $\tilde\sigma:x\mapsto \sqrt{\epsilon\delta\mu^{-1}}\sigma(x)$, which belongs to $K(U,U)$ and satisfies $\tilde\sigma^2=\epsilon\delta\id$. The result follows now from Proposition \ref{pr:sigma_in_K}.
\end{proof}

This Corollary allows us to give examples of $(1,1)$ Freudenthal-Kantor triple systems starting from structurable algebras:

\begin{example}\label{ex:structurable_11}
Let $(A,\cdot,\bar{\ })$ be a structurable algebra, and assume there is an element $f\in A$ with $\bar f=-f$ and $0\ne f^{\cdot 2}\in \bF 1$. Write $f^{\cdot 2}=\mu 1$. Note that this is always the case for the simple structurable algebras of skew-dimension one \cite[Lemma 2.1(b)]{AllisonFaulkner_CayleyDickson}. Consider the associated Kantor triple system (that is, $(-1,1)$ Freudenthal-Kantor triple system), with triple product as in \eqref{eq:xyzVxyz}. Then $K(f,1)x=\{fx1\}-\{1xf\}=V_{f,x}(1)-V_{1,x}(f)=(f-\bar f)\cdot x=2f\cdot x$ (see \eqref{eq:Vxy}). But \eqref{eq:str1} gives $f\cdot(f\cdot x)=f^{\cdot 2}\cdot x=\mu x$. Hence, with $\sigma(x)=f\cdot x$, we are in the situation of Corollary \ref{co:sigma_in_K}, and therefore, with the new triple product given by $\{xyz\}^\sim =\{x(f\cdot y)z\}=V_{x,f\cdot y}(z)$, $A$ becomes a $(1,1)$ Freudenthal-Kantor triple system. \qed
\end{example}

\bigskip

If $(U,xyz)$ is a nontrivial ($U\ne 0$) left unital $(\epsilon,\delta)$ Freudenthal-Kantor triple system, that is, there is an element $e$ such that $L(e,e)=\id$, then \eqref{eq:FK1} with $u=v=e$ gives $(1+\epsilon)L(x,y)=0$ for any $x,y\in U$. Therefore $\epsilon=-1$. Hence only $(-1,\delta)$ Freudenthal-Kantor triple systems may be left unital.

Most of the arguments in the previous sections work for $(-1,-1)$ Freudenthal-Kantor triple systems, so Lemma \ref{le:rhomu} is valid for them. Lemma \ref{le:rho13} has to be changed to the next result, whose proof is obtained following the same arguments step by step.

\begin{lemma}\label{le:rho1-1}
Let $(U,xyz)$ be a left unital $(-1,-1)$ Freudenthal-Kantor triple system with left unit $e$. Then, for any $u,v\in U$, the following conditions hold:
\begin{gather}
K(u,e)e=(\rho+\id)(u),\label{eq:2a'}\\
K(u,v)=\frac{1}{2}K(K(u,v)e,e), \label{eq:2b'}\\
\rho^2=\id. \label{eq:2c'}
\end{gather}
In particular, $\rho$ and $\mu$ are invertible linear operators, and $\mu^2=\id$. \qed
\end{lemma}

We want to prove a result analogous to Lemma \ref{le:UUsigma_e}, which shows that we can modify slightly the triple product of a left unital Kantor triple system with the help of a suitable involutive automorphism, and get a new left unital Kantor triple system with stronger restrictions on the left unit.

We could follow a path parallel to the one for left unital Kantor triple systems, but Proposition \ref{pr:sigma_in_K} allows a more direct approach.

\begin{theorem}\label{th:FKTSmu-1-1}
Let $(U,xyz)$ be a left unital $(-1,-1)$ Freudenthal-Kantor triple system with left unit $e$. Then $\mu:x\mapsto exe$, is an involutive automorphism. Besides, consider the new $(-1,-1)$ Freudenthal-Kantor triple system $(U,\{xyz\})$ with $\{xyz\}=x\mu(y)z$. Then $\{eex\}=x=\{exe\}$ for any $x\in U$.

Conversely, let $(U,\{xyz\})$ be a $(-1,-1)$ Freudenthal-Kantor triple system containing an element $e$ such that $\{eex\}=x=\{exe\}$ for any $x\in U$, and endowed with an involutive automorphism $\sigma$ such that $\sigma(e)=e$. Define a new triple product on $U$ by $xyz=\{x\sigma(y)z\}$. Then $(U,xyz)$ is a left unital $(-1,-1)$ Freudenthal-Kantor triple system with left unit $e$, $\sigma$ is an automorphism of $(U,xyz)$, and $\sigma=\mu:x\mapsto exe$.
\end{theorem}
\begin{proof}
Since $K(e,e)x=exe+exe=2\mu(x)$ for any $x\in U$, it follows that $\mu$ belongs to $K(U,U)$, and Lemma \ref{le:rho1-1} shows that $\mu^2=\id$. Hence Proposition \ref{pr:sigma_in_K} shows that $\mu$ is an automorphism.
Now the first part of the Theorem follows since $\{exe\}=e\mu(x)e=\mu^2(x)=x$ (Lemma \ref{le:rho1-1}).

For the converse, just note that $\sigma(x)=\{e\sigma(x)e\}=exe$ for any  $x\in U$.
\end{proof}

\begin{definition}[{\cite[Definition 3.3]{dicyclic}}]\label{df:special}
An $(\epsilon,\delta)$ Freudenthal-Kantor triple system $(U,xyz)$ is said to be \emph{special} in case
\begin{equation}\label{eq:special}
K(x,y)=\epsilon\delta L(y,x)-\epsilon L(x,y)
\end{equation}
holds for any $x,y\in U$.

Moreover, $(U,xyz)$ is said to be \emph{unitary} in case the identity map belongs to $K(U,U)$ (the linear span of the endomorphisms $K(x,y)$).
\end{definition}

If an $(\epsilon,\delta)$ Freudenthal-Kantor triple system is unitary, then necessarily $\epsilon=\delta$, and the system is special (see \cite[Proposition 3.4]{dicyclic}.)

\begin{corollary}\label{co:FKTSmu-1-1}
Let $(U,xyz)$ be a left unital $(-1,-1)$ Freudenthal-Kantor triple system with left unit $e$. Consider the new $(-1,-1)$ Freudenthal-Kantor triple system $(U,\{xyz\})$ with $\{xyz\}=x\mu(y)z$. Then $(U,\{xyz\})$ is a unitary, and hence special, $(-1,-1)$ Freudenthal-Kantor triple system.
\end{corollary}
\begin{proof}
Denote by $\tilde L(x,y)$ and $\tilde K(x,y)$ the $L$ and $K$ operators in $(U,\{x,yz\})$. Then $\tilde K(e,e):x\mapsto 2\{exe\}=2x$, and hence $\tilde K(e,e)=2\id$ and $(U,\{xyz\})$ is unitary.
\end{proof}


\section{Associated Lie algebras and superalgebras}

Let $(U,xyz)$ be an $(\epsilon,\delta)$ Freudenthal-Kantor triple system, then the space of $2\times 1$ matrices over $U$:
\begin{equation}\label{eq:(anti)Lie}
\calT=\calT(U,xyz)=\left\{\begin{pmatrix}x\\ y\end{pmatrix}: x,y\in U\right\}
\end{equation}
becomes a Lie triple system for $\delta=1$ and an anti-Lie triple system for $\delta=-1$ (see \cite[Section 3]{YO84}) by means of the triple product:
\begin{equation}\label{eq:Lietriple}
\begin{split}
&\left[\begin{pmatrix}a_1\\ b_1\end{pmatrix}\begin{pmatrix} a_2\\ b_2\end{pmatrix}\begin{pmatrix} a_3\\ b_3\end{pmatrix}\right]\\
&\qquad =
\begin{pmatrix} L(a_1,b_2)-\delta L(a_2,b_1)&\delta K(a_1,a_2)\\
-\epsilon K(b_1,b_2)&\epsilon L(b_2,a_1)-\epsilon\delta L(b_1,a_2)\end{pmatrix}\begin{pmatrix} a_3\\ b_3\end{pmatrix}
\end{split}
\end{equation}
and, therefore, the vector space
\begin{equation}\label{eq:L}
\calL=\espan{\begin{pmatrix} L(a,b)&K(c,d)\\ K(e,f)&\epsilon L(b,a)\end{pmatrix} : a,b,c,d,e,f\in U}
\end{equation}
is a Lie subalgebra of $\Mat_2\bigl(\End_\bF(U)\bigr)^-$. (Given an associative algebra $A$, $A^-$ denotes the Lie algebra defined on $A$ with product given by the usual Lie bracket $[x,y]=xy-yx$.)

Hence we get either a $\bZ_2$-graded Lie algebra (for $\delta=1$) or a Lie superalgebra (for $\delta=-1$)
\begin{equation}\label{eq:gU}
\frg(U)=\frg(U,xyz)=\calL\oplus\calT
\end{equation}
where $\calL$ is the even part and $\calT$ the odd part. The bracket in $\frg(U)$ is given by:
\begin{itemize}
\item the given bracket in $\calL$ as a subalgebra of $\Mat_2\bigl(\End_\bF(U)\bigr)^-$,
\item $[M,X]=M(X)$ for any $M\in \calL$ and $X\in \calT$ (note that $\Mat_2\bigl(\End_\bF(U)\bigr)\simeq \End_\bF(\calT)$),
\item for any $a_1,a_2,b_1,b_2\in U$:
\[
\left[\begin{pmatrix}a_1\\ b_1\end{pmatrix},\begin{pmatrix}a_2\\ b_2\end{pmatrix}\right]=\begin{pmatrix} L(a_1,b_2)-\delta L(a_2,b_1)&\delta K(a_1,a_2)\\
-\epsilon K(b_1,b_2)&\epsilon L(b_2,a_1)-\epsilon\delta L(b_1,a_2)\end{pmatrix}.
\]
\end{itemize}

To simplify things, we will talk about the (anti-)Lie triple system $\calT$ and the Lie (super)algebra $\frg(U)$, meaning that they are a Lie triple system and a Lie algebra for $\delta=1$ and an anti-Lie triple system and a Lie superalgebra for $\delta=-1$.

This (super)algebra $\frg(U)$ is $\bZ$-graded as follows:
\[
\begin{split}
\frg(U)_{(0)}&=\espan{\begin{pmatrix} L(a,b)&0\\ 0&\epsilon L(b,a)\end{pmatrix} : a,b\in U},\\
\frg(U)_{(1)}&=\begin{pmatrix}U\\ 0\end{pmatrix},\\
\frg(U)_{(-1)}&=\begin{pmatrix} 0\\ U\end{pmatrix},\\
\frg(U)_{(2)}&=\espan{\begin{pmatrix} 0&K(a,b)\\ 0&0\end{pmatrix} : a,b\in U},\\
\frg(U)_{(-2)}&=\espan{\begin{pmatrix} 0&0\\ K(a,b)&0\end{pmatrix} : a,b\in U},
\end{split}
\]
so that $\frg(U)$ is $5$-graded and
\[
\calL=\frg(U)_{(-2)}\oplus\frg(U)_{(0)}\oplus\frg(U)_{(2)},\qquad \calT=\frg(U)_{(-1)}\oplus\frg(U)_{(1)}.
\]

This Lie (super)algebra $\frg(U)$ is completely determined by the (anti-)Lie triple system $\calT$. In case $(U,xyz)$ is the Kantor triple system defined on a structurable algebra $(A,\cdot,\bar{\ })$ by means of \eqref{eq:xyzVxyz}, then $\frg(U)$ coincides with the $5$-graded Lie algebra $\calK(A,\bar{\ })$ defined in \cite{Allison79}.

\begin{proposition}\label{pr:gUgUsigma}
Let $\sigma$ be an involutive automorphism of an $(\epsilon,\delta)$ Freudenthal-Kantor triple system. Then, with the new triple product defined by $\{xyz\}=x\sigma(y)z$ for any $x,y,z\in U$, $(U,\{xyz\})$ is again an $(\epsilon,\delta)$ Freudenthal-Kantor triple system and the associated Lie (super)algebras $\frg(U,xyz)$ and $\frg(U,\{xyz\})$ are isomorphic (as $5$-graded Lie (super)algebras).
\end{proposition}
\begin{proof}
The proof of Lemma \ref{le:UUsigma} applies here and shows that $(U,\{xyz\})$ is an $(\epsilon,\delta)$ Freudenthal-Kantor triple system.

Consider now the (anti-)Lie triple systems $\calT=\calT(U,xyz)$ and $\tilde\calT=\calT(U,\{xyz\})$. Denote by $[...]$ the triple product in $\calT$ and by $[...]^\sim$ the one in $\tilde\calT$. Let $\Phi:\calT\rightarrow \tilde\calT$ the linear isomorphism given by
\[
\Phi\begin{pmatrix} a\\ b\end{pmatrix}=\begin{pmatrix} a\\ \sigma(b)\end{pmatrix},
\]
for any $a,b\in U$. Then, for any $a_i,b_i\in U$, $i=1,2,3$, we have:
\[
\begin{split}
\Phi&\left(\left[\begin{pmatrix}a_1\\ b_1\end{pmatrix}\begin{pmatrix} a_2\\ b_2\end{pmatrix}\begin{pmatrix} a_3\\ b_3\end{pmatrix}\right]\right)\\
 &=
  \Phi\begin{pmatrix} L(a_1,b_2)a_3-\delta L(a_2,b_1)a_3+\delta K(a_1,a_2)b_3\\
  -\epsilon K(b_1,b_2)a_3+\epsilon L(b_2,a_1)b_3-\epsilon\delta L(b_1,a_2)b_3\end{pmatrix}\\
 &=\begin{pmatrix} L(a_1,b_2)a_3-\delta L(a_2,b_1)a_3+\delta K(a_1,a_2)b_3\\
  -\epsilon K(\sigma(b_1),\sigma(b_2))\sigma(a_3)+\epsilon L(\sigma(b_2),\sigma(a_1))\sigma(b_3)-\epsilon\delta L(\sigma(b_1),\sigma(a_2))\sigma(b_3)\end{pmatrix}\\
 &=\begin{pmatrix} \tilde L(a_1,\sigma(b_2))a_3-\delta \tilde L(a_2,\sigma(b_1))a_3+\delta \tilde K(a_1,a_2)\sigma(b_3)\\
  -\epsilon \tilde K(\sigma(b_1),\sigma(b_2))a_3+\epsilon \tilde L(\sigma(b_2),a_1)\sigma(b_3)-\epsilon\delta\tilde L(\sigma(b_1),a_2)\sigma(b_3)\end{pmatrix}\\
 &=\left[\begin{pmatrix}a_1\\ \sigma(b_1)\end{pmatrix}\begin{pmatrix} a_2\\ \sigma(b_2)\end{pmatrix}\begin{pmatrix} a_3\\ \sigma(b_3)\end{pmatrix}\right]^\sim
 =\left[\Phi\begin{pmatrix}a_1\\ b_1\end{pmatrix}\,\Phi\begin{pmatrix} a_2\\ b_2\end{pmatrix}\,\Phi\begin{pmatrix} a_3\\ b_3\end{pmatrix}\right]^\sim.
\end{split}
\]
Since $\frg(U,xyz)$ and $\frg(U,\{xyz\})$ are determined by $\calT$ and $\tilde\calT$, the result follows.
\end{proof}

Theorem \ref{th:main} and Proposition \ref{pr:gUgUsigma} have the next direct consequence (see also \cite{BenkartSmirnov}):

\begin{corollary}
Let $(U,xyz)$ be a left unital Kantor triple system. Then its Lie algebra $\frg(U,xyz)$ is isomorphic, as a $5$-graded Lie algebra, to the Lie algebra $\calK(A,\bar{\ })$ defined in \cite{Allison79} for a structurable algebra $(A,\cdot,\bar{\ })$. In particular, $\frg(U,xyz)$ is a Lie algebra graded over the nonreduced root system $BC_1$ and of type $B_1$.
\end{corollary}

Special $(-1,-1)$ Freudenthal-Kantor triple systems give rise to strictly $BC_1$-graded Lie superalgebras of type $C_1$ (see \cite[Corollary 4.6]{dicyclic}).
For left unital $(-1,-1)$ Freudenthal-Kantor triple systems, we can strengthen this result, as we obtain superalgebras graded by a (Lie superalgebra) root system. The reader is referred to \cite{BenkartElduqueBmn} for the results needed on Lie superalgebras graded over the root systems of the simple classical Lie superalgebras of type $B(m,n)$.

The $B(0,1)$-graded Lie superalgebras are the Lie superalgebras that contain a subalgebra isomorphic to the orthosymplectic Lie superalgebra $B(0,1)=\frosp(1,2)$, and such that, as a module for this subalgebra, they are a sum of copies of the adjoint module or the module obtained from it by interchanging the even and odd parts (the adjoint module with the parity changed), plus copies of the natural module or this module with the parity changed, plus a submodule with trivial action \cite[Section 6]{BenkartElduqueBmn}. In case the parity of any of the copies of the adjoint and natural modules that appear in this decomposition is the natural one, we will say that the Lie superalgebra is a \emph{strictly $B(0,1)$-graded Lie superalgebra}. In this case, the coordinate superalgebra considered in \cite{BenkartElduqueBmn} is actually an algebra (the odd part is zero).

\begin{corollary}\label{co:gUsuper_rootgraded}
Let $(U,xyz)$ be a left unital $(-1,-1)$ Freudenthal-Kantor triple system. Then its Lie superalgebra $\frg(U,xyz)$ is a strictly $B(0,1)$-graded Lie superalgebra.
\end{corollary}
\begin{proof}
By Propositions \ref{pr:gUgUsigma} and \ref{th:FKTSmu-1-1}, we may assume that there is an element $e\in U$ such that $eex=x=exe$ for any $x\in U$. Thus $L(e,e)=\id$ and $K(e,e)=2\id$. Then the subalgebra of the Lie superalgebra $\frg(U,xyz)$ generated by the odd elements $\left(\begin{smallmatrix} e\\ 0\end{smallmatrix}\right)$ and $\left(\begin{smallmatrix} 0\\ e\end{smallmatrix}\right)$ is
\[
\frh=\espan{\begin{pmatrix} 0&0\\ \id&0\end{pmatrix},
 \begin{pmatrix} 0\\ e\end{pmatrix},
 \begin{pmatrix} \id&0\\ 0&-\id\end{pmatrix},
 \begin{pmatrix} e\\ 0\end{pmatrix},
 \begin{pmatrix} 0&\id\\ 0&0\end{pmatrix}},
\]
which is isomorphic to the simple orthosymplectic Lie superalgebra $B(0,1)=\frosp(1,2)$, whose even part is isomorphic to $\frsl_2(\bF)$ and its odd part is the two-dimensional natural irreducible module for its even part.

Given any $x\in U_1\cup U_{-1}$, the $\frh$-submodule $\frm_x$ of $\frg(U,xyz)$ generated by $\left(\begin{smallmatrix} 0\\ x\end{smallmatrix}\right)$ is the linear span of the elements
\[
\begin{pmatrix} 0&0\\ K(e,x)&0\end{pmatrix},
 \begin{pmatrix} 0\\ x\end{pmatrix},
 \begin{pmatrix} L(e,x)&0\\ 0&-L(x,e)\end{pmatrix},
 \begin{pmatrix} x\\ 0\end{pmatrix},
 \begin{pmatrix} 0&K(e,x)\\ 0&0\end{pmatrix},
 \]
because $L(e,x)e=x$, $L(x,e)e=\rho(x)=\pm x$ and $K(e,x)e=eex+xee=x+\rho(x)$, which is $0$ for $x\in U_{-1}$ and $2x$ for $x\in U_1$. Equation \eqref{eq:2b'} shows then than $K(e,x)=0$ for $x\in U_{-1}$. Therefore, if $x\in U_{-1}$, then $\frm_x$  is the three-dimensional natural module for $\frh\cong\frosp(1,2)$ (with the natural parity), while for $x\in U_1$, this module is isomorphic to the adjoint module for $\frh$ (again with the natural parity).

Equation \eqref{eq:2b'} gives $K(U,U)=K(U,e)=K(U_1,e)$. Therefore, as a module for $\frh$, $\frg(U,xyz)$ is the sum of the adjoint modules $\frm_x$ for $x\in U_1$, the natural modules $\frm_x$ for $x\in U_{-1}$, and the submodule with trivial action of $\frh$ given by:
\[
\left\{\begin{pmatrix} \varphi&0\\ 0&-\varphi\end{pmatrix} : \varphi\in L(U,U),\ \varphi(e)=0\right\}.
\]
This shows that $\frg(U,xyz)$ is $B(0,1)$-graded \cite{BenkartElduqueBmn}.
\end{proof}

We finish the paper with a converse to Corollary \ref{co:gUsuper_rootgraded}. But first we need some preliminaries.

Let $(U,xyz)$ be an $(\epsilon,\delta)$ Freudenthal-Kantor triple system, and consider the (anti-)Lie triple system $\calT$ in \eqref{eq:(anti)Lie} as well as the $5$-graded Lie (super)algebra $\frg(U)$ in \eqref{eq:gU}. The linear map $\Phi:\calT\rightarrow \calT$ given by
\[
\Phi:\begin{pmatrix} x\\ y\end{pmatrix}\mapsto \begin{pmatrix}y\\ -\epsilon\delta x\end{pmatrix}
\]
is easily checked to be an automorphism of the (anti-)Lie triple system $\calT$, such that $\Phi^2=-\epsilon\delta\id$, and hence it induces an automorphism of $\frg(U)$, also denoted by $\Phi$, which satisfies $\Phi\bigl((\frg(U)_{(i)}\bigr)=\frg(U)_{(-i)}$, for any  $i=0,\pm 1,\pm 2$, and $\Phi^2=\id$ if $\epsilon\delta=-1$, while $\Phi^2$ is the automorphism whose restriction to $\frg(U)_{(i)}$ is $(-1)^i\id$ for $i=0,\pm 1,\pm 2$. (For $\delta=-1$, this is the grading automorphism of the Lie superalgebra $\frg(U)$.)

We may identify $\frg(U)_{(1)}$ with $U$ by identifying $\left(\begin{smallmatrix} x\\ 0\end{smallmatrix}\right)\in \frg(U)_{(1)}$ with $x\in U$. Then the triple product on $U$ is recovered as:
\begin{equation}\label{eq:xyzLie}
xyz=[[x,\Phi(y)],z]
\end{equation}
for any $x,y,z\in U$, where on the right hand side we use the Lie bracket in $\frg(U)$.

Conversely, take $\epsilon,\delta$ equal to $1$ or $-1$ and let $\frg$ be a $5$-graded Lie algebra for $\delta=1$, or a consistently $5$-graded Lie superalgebra for $\delta=-1$ (this means that $\frg\subo=\frg_{-2}\oplus\frg_0\oplus\frg_2$ and $\frg\subuno=\frg_{-1}\oplus\frg_1$). Assume, moreover, that $\frg$ is endowed with an automorphism $\Phi$ such that $\Phi(\frg_i)=\frg_{-i}$ for $i=0,\pm 1,\pm 2$, and $\Phi^2=\id$ if $\epsilon\delta=-1$, while $\Phi^2$ is the automorphism whose restriction to $\frg_i$ is $(-1)^i\id$ for $i=0,\pm 1,\pm 2$. On $U=\frg_1$ define the triple product by formula \eqref{eq:xyzLie}. Then the operators $L(x,y):z\mapsto xyz$ and $K(x,y):z\mapsto xzy-\delta yzx$ are given by:
\begin{equation}\label{eq:LxyKxy_ad}
\begin{split}
L(x,y)z&=[[x,\Phi(y)],z]=\ad_{[x,\Phi(y)]}(z),\\[4pt]
K(x,y)z&=[[x,\Phi(z)],y]-\delta[[y,\Phi(z)],x]\\
    &=\delta[x,[y,\Phi(z)]]-[y,[x,\Phi(z)]]\\
    &=\delta[[x,y],\Phi(z)]=\delta\ad_{[x,y]}(\Phi(z)).
\end{split}
\end{equation}
(Note that $\ad_{[x,y]}=\ad_x\ad_y-\delta\ad_y\ad_x$, because for $\delta=-1$ the elements in $\frg_{\pm 1}$ are odd, while for $\delta=1$ this is clear.)

Therefore, for $u,v,x,y,z\in U$ we compute:
\[
\begin{split}
[L(u,v),&L(x,y)](z)\\
    &=[\ad_{[u,\Phi(v)]},\ad_{[x,\Phi(y)]}](z)\\
    &=\Bigl(\ad_{[\ad_{[u,\Phi(v)]}(x),\Phi(y)]}
       +\ad_{[x,[[u,\Phi(v)],\Phi(y)]]}\Bigr)(z)\\
    &=\Bigl(\ad_{[\ad_{[u,\Phi(v)]}(x),\Phi(y)]}
       -\delta\ad_{[x,\Phi([[v,\Phi^{-1}(u)],y])]}\Bigr)(z)\\
    &=\Bigl(\ad_{[\ad_{[u,\Phi(v)]}(x),\Phi(y)]}
       +\epsilon\ad_{[x,\Phi([[v,\Phi(u)],y])]}\Bigr)(z)\ \text{(as $\Phi^2(u)=-\epsilon\delta u$)}\\
    &=L(uvx,y)z+\epsilon L(x,vuy)z,
\end{split}
\]
and
\[
\begin{split}
K(K(u,v)x,y)z&=\delta K([[u,v],\Phi(x)],y)z\\
    &=[[[[u,v],\Phi(x)],y],\Phi(z)]=[[[u,v],[\Phi(x),y]],\Phi(z)],
\end{split}
\]
because $[[u,v],y]\in [[\frg_1,\frg_1],\frg_1]\subseteq [\frg_2,\frg_1]=0$, while
\[
\begin{split}
L(y,x)K(u,v)z&-\epsilon K(u,v)L(x,y)z\\
    &=\delta[[y,\Phi(x)],[[u,v],\Phi(z)]]-\epsilon\delta[[u,v],\Phi([[x,\Phi(y)],z])\\
    &=-[[\Phi(x),y],[[u,v],\Phi(z)]]+[[u,v],[[\Phi(x),y],\Phi(z)]]\\
    &=[[[u,v],[\Phi(x),y]],\Phi(z)].
\end{split}
\]
Hence \eqref{eq:FK1} and \eqref{eq:FK2} are satisfied, and $(U,xyz)$ is an $(\epsilon,\delta)$ Freudenthal Kantor triple system.

\begin{example}\label{ex:FTS}
Let $\frg$ be a finite dimensional simple Lie algebra of rank $l$ over an algebraically closed field $\bF$ of characteristic zero. Fix a Cartan subalgebra $\frh$ and denote by $\Sigma$ the set of roots. Thus
\[
\frg=\frh\oplus\bigl(\oplus_{\alpha\in\Sigma}\frg_\alpha\bigr).
\]
Let $\Delta$ be a system of simple roots, which splits $\Sigma$ into positive and negative roots $\Sigma=\Sigma^+\cup\Sigma^-$, where $\Sigma^+$ is the set of roots that are sums of simple roots and $\Sigma^-=-\Sigma^+$. Take a Chevalley basis $\{x_\alpha,\,\alpha\in\Sigma;\, h_i,\, 1\leq i\leq l\}$ (see \cite[\S 25]{Humphreys}). In particular, $[x_\alpha,x_{-\alpha}]=h_\alpha$ with $\alpha(h_\alpha)=2$ for any $\alpha\in\Sigma$. Consider the order two automorphism $\Phi$ determined by $\Phi(x_\alpha)=-x_{-\alpha}$ for any $\alpha\in \Sigma$ (so $\Phi\vert_\frh=-\id$). Let $\rho$ be the highest root. Then $\frg$ is $5$-graded: $\frg=\frg_{-2}\oplus\frg_{-1}\oplus\frg_0\oplus\frg_1\oplus\frg_2$, with
\[
\begin{split}
\frg_{\pm 2}&=\bF x_{\pm\rho},\\
\frg_{\pm 1}&=\oplus \{\frg_\alpha :\alpha\in\Sigma^{\pm},\, \alpha\ne\pm\rho,\, (\rho\vert\alpha)\ne 0\},\\
\frg_0&=\frh\oplus\bigl(\oplus\{\frg_\alpha: \alpha\in\Sigma,\, (\rho\vert\alpha)=0\}\bigr),
\end{split}
\]
where $(.\vert .)$ is the bilinear form induced by the Killing form.

Equation \eqref{eq:xyzLie} endows $U=\frg_1$ with the structure of a Kantor triple system ($(-1,1)$ Freudenthal-Kantor triple system), by means of
\[
uvw=[[u,\Phi(v)],w].
\]

Moreover, since $\dim \frg_2=1$, $[u,v]=\langle u\vert v\rangle x_\rho$, for a skew-symmetric bilinear form $\langle .\vert .\rangle$. Now \eqref{eq:LxyKxy_ad} gives:
\[
K(u,v)w=[[u,v],\Phi(w)]=\langle u\vert v\rangle [x_\rho,\Phi(w)].
\]
Define $\sigma\in\End_\bF(U)$ by $\sigma(z)=[x_\rho,\Phi(z)]$, so we get $K(u,v)=\langle u\vert v\rangle\sigma$ for any $u,v\in U=\frg_1$. Especially $\sigma\in K(U,U)$. Besides, for any $z\in U$:
\[
\begin{split}
\sigma^2(z)&=[x_\rho,\Phi(\sigma(z))]=[x_\rho,\Phi([x_\rho,\Phi(z)])]\\
    &=[x_\rho,[\Phi(x_\rho),\Phi^2(z)]]=[x_\rho,[-x_{-\rho},z]]\\
    &=-[[x_\rho,x_{-\rho}],z]\quad\text{(because $[x_\rho,z]\in[\frg_2,\frg_1]=0$)}\\
    &=-[h_\rho,z]=-z,
\end{split}
\]
where we have used that $\alpha(h_\rho)=1$ for any $\alpha\in\Sigma^+$ with $(\rho\vert\alpha)\ne 0$. Therefore, we have $\sigma^2=-\id$, and Corollary \ref{co:sigma_in_K} shows that $\sigma$ is an automorphism of our Kantor triple system $(U,uvw)$. Moreover, with the new triple product
\[
\{uvw\}=u\sigma(v)w,
\]
$(U,\{uvw\})$ is a $(1,1)$ Freudenthal-Kantor triple system. Besides, if we denote by $L^*$ and $K^*$ the $L$ and $K$ operators for this new triple system, $K^*(u,v)=K(u,v)\sigma=\langle u\vert v\rangle\sigma^2=-\langle u\vert v\rangle\id$, so this system is \emph{balanced} (its $K$ operators are given by a bilinear form).

Finally, with the triple product
\[
(uvw)=\{uvw\}-\frac{1}{2}\langle u\vert v\rangle w+\frac{1}{2}\langle u\vert w\rangle v+\frac{1}{2}\langle v\vert w\rangle u,
\]
$U$ becomes a Freudenthal triple system (see \cite{Meyberg}). \qed
\end{example}

\bigskip

The orthosymplectic Lie superalgebra $B(0,1)=\frosp(1,2)$ is the subalgebra of the general linear Lie superalgebra $\frgl(1,2)$ given by
\[
\frb=\left\{\left(\begin{array}{c|cc} 0&\mu&\nu\\
\hline   -\nu &\alpha&\beta\\ \mu&\gamma&-\alpha\end{array}\right)\ :\ \alpha,\beta,\gamma,\mu,\nu\in \bF\right\}.
\]
A natural basis of $\frb$ consists of the elements
\begin{gather*}
H=\left(\begin{array}{c|cc} 0&0&0\\
\hline   0 &1&0\\ 0&0&-1\end{array}\right),\quad
E=\left(\begin{array}{c|cc} 0&0&0\\
\hline   0 &0&1\\ 0&0&0\end{array}\right),\quad
F=\left(\begin{array}{c|cc} 0&0&0\\
\hline   0 &0&0\\ 0&1&0\end{array}\right),\\
X=\left(\begin{array}{c|cc} 0&0&-1\\
\hline   1 &0&0\\ 0&0&0\end{array}\right),\quad
Y=\left(\begin{array}{c|cc} 0&1&0\\
\hline   0 &0&0\\ 1&0&0\end{array}\right),
\end{gather*}
where $\{H,E,F\}$ form a basis of the even part, which is isomorphic to $\frsl_2(\bF)$, and $\{X,Y\}$ of the odd part.

Its natural three-dimensional module is isomorphic to the following subspace of $\frgl(1,2)$:
\[
\frs=\left\{\left(\begin{array}{c|cc} 2\alpha&\mu&\nu\\
\hline   \nu &\alpha& 0\\ -\mu& 0 &\alpha\end{array}\right)\ :\ \alpha,\mu,\nu\in \bF\right\}.
\]
A natural basis of this module consists of the following elements:
\[
\hat H=\left(\begin{array}{c|cc} 2&0&0\\
\hline   0 &1&0\\ 0&0&1\end{array}\right),\quad
\hat X=\left(\begin{array}{c|cc} 0&0&1\\
\hline   1 &0&0\\ 0&0&0\end{array}\right),\quad
\hat Y=\left(\begin{array}{c|cc} 0&-1&0\\
\hline   0 &0&0\\ 1&0&0\end{array}\right),
\]
where $\hat H$ is even and $\hat X$ and $\hat Y$ are odd.

The automorphism $\Phi$ of $\frgl(1,2)$ given by conjugation by the matrix
\[
\left(\begin{array}{c|cc} 1&0&0\\
\hline   0 &0&-1\\ 0&1&0\end{array}\right)
\]
leaves invariant both $\frb$ and $\frh$. Its square is the grading automorphism of $\frgl(1,2)$. That is, it equals $\id$ on the even part and $-\id$ on the odd part. A simple computation shows:
\begin{equation}\label{eq:PhiX_hatX}
\Phi(X)=Y,\qquad\Phi(\hat X)=\hat Y.
\end{equation}

If $\frg$ is a $B(0,1)$-graded Lie superalgebra, then it contains a subalgebra isomorphic to $\frb$, and \cite[Theorem 6.20]{BenkartElduqueBmn} shows that, up to isomorphism, $\frg$ decomposes, as a module for $\frb$, as
\[
\frg=\bigl(\frb\otimes \calA\bigr)\oplus \bigl(\frs\otimes \calB\bigr)\oplus \calD,
\]
for suitable vector superspaces $\calA$ and $\calB$. Besides, the subalgebra isomorphic to $\frb$ is identified with $\frb\otimes 1$, for a distinguished even element $1\in \calA$, and $\calD$ is the centralizer in $\frg$ of this subalgebra. The action of $H$ provides a $5$-grading of $\frg$, with $\frg_i=\{ Z\in \frg: [H,Z]=iZ\}$. Hence we have:
\begin{equation}\label{eq:g_5graded}
\begin{split}
\frg_{-2}&=F\otimes \calA,\\
\frg_{-1}&=(Y\otimes\calA)\oplus (\hat Y\otimes \calB),\\
\frg_0&=(H\otimes\calA)\oplus (\hat H\otimes \calB)\oplus \calD,\\
\frg_1&=(X\otimes \calA)\oplus (\hat X\otimes\calB),\\
\frg_2&= E\otimes\calA.
\end{split}
\end{equation}
The Lie superalgebra $\frg$ is strictly $B(0,1)$-graded if and only if the superspaces $\calA$ and $\calB$ have trivial odd part, so they are standard vector spaces.

The automorphism $\Phi$ of $\frgl(1,2)$ above extends to an automorphism of $\frg$, which acts just on $\frb$ and $\frh$, and which will be denoted by $\Phi$ too. Therefore, if $\frg$ is strictly $B(0,1)$-graded, the vector space $U=\frg_1$ is a $(-1,-1)$ Freudenthal-Kantor triple system with the triple product given by $xyz=[[x,\Phi(y)],z]$

Now the promised converse of Corollary \ref{co:gUsuper_rootgraded} is easy:

\begin{theorem}\label{th:B01_graded}
Let $\frg$ be a strictly $B(0,1)$-graded Lie superalgebra. Consider the $5$-grading given by the action of the element $H$ above. Then the element $e=X\otimes 1$ is a left unit of the $(-1,-1)$ Freudenthal-Kantor triple system $(U=\frg_1,xyz)$ which satisfies $exe=x$ for any $x\in U$.
\end{theorem}
\begin{proof}
We must prove that $eex=exe=x$ for any $x\in \frg_1=(X\otimes\calA)\oplus(\hat X\otimes \calB)$. This amounts to prove that $L(e,e)=\id$ and $K(e,e)=2\id$.

But \eqref{eq:LxyKxy_ad} shows that $L(e,e)x=\ad_{[e,\Phi(e)]}(x)$ and $K(e,e)x=-\ad_{[e,e]}(\Phi(x))$. Note that $\Phi(X)=Y$ by \eqref{eq:PhiX_hatX}, so $[e,\Phi(e)]=[X,Y]\otimes 1=H\otimes 1$, which acts as the identity on $\frg_1$, as the $5$-grading is given precisely by the eigenspace decomposition relative to $H\otimes 1$. Hence $L(e,e)=\id$.

On the other hand, $[X,X]=2X^2=-2E$, so $[e,e]=-2E\otimes 1$, and hence $K(e,e)(X\otimes a)=2[E\otimes 1,\Phi(X)\otimes a]=2[E\otimes 1,Y\otimes a]=2[E,Y]\otimes a=2X\otimes a$, while $K(e,e)(\hat X\otimes b)=2[E\otimes 1,\Phi(\hat X)\otimes b]=2[E\otimes 1,\hat Y\otimes b]=2[E,\hat Y]\otimes b=2\hat X\otimes b$. Hence $K(e,e)=2\id$.
\end{proof}


\end{document}